\documentclass{amsart}

\usepackage[latin1]{inputenc}
\usepackage{subfigure}
\usepackage{color}
\usepackage{amsmath, latexsym, amssymb,longtable,comment}
\usepackage{graphicx}
\usepackage{amsfonts}
\usepackage{amsmath}
\usepackage{graphics}
\usepackage{amssymb}
\usepackage{latexsym}
\usepackage{amscd}
\usepackage{todonotes}
\usepackage[all]{xy}
\usepackage{bm}
\usepackage{enumerate}
\usepackage{todonotes}
\usepackage{mathrsfs}
\usepackage{hyperref}
\usepackage[capitalize,nameinlink,noabbrev,nosort]{cleveref}
\usepackage{array}
\RequirePackage[maxbibnames=99,doi=false,isbn=false,url=false,sorting=nyt,giveninits,maxnames=10,backend=bibtex]{biblatex}
\def\OM#1{(\textcolor{orange}{OM:#1})}

\theoremstyle{plain}
\newtheorem{theorem}{Theorem}[section]
\newtheorem{theorem*}{Theorem}

\newtheorem{lemma}[theorem]{Lemma}

\newtheorem{corollary}[theorem]{Corollary}

\crefname{theorem}{Theorem}{Theorems}
\crefname{conjecture}{Conjecture}{Conjectures}
\crefname{lemma}{Lemma}{Lemmas}

\theoremstyle{definition}

\newtheorem{example}[theorem]{Example}

\newtheorem{remark}[theorem]{Remark}
\numberwithin{equation}{section}
\numberwithin{figure}{section}
\numberwithin{table}{section}

\DeclareMathOperator{\ST}{ST}

\newcommand{\T}{\mathcal{T}}
\newcommand{\Z}{\mathbb{Z}}

\DeclareMathOperator{\arm}{arm}
\DeclareMathOperator{\leg}{leg}
\DeclareMathOperator{\Des}{Des}

\DeclareMathOperator{\maj}{maj}
\DeclareMathOperator{\coinv}{coinv}

\DeclareMathOperator{\wt}{wt}

\DeclareMathOperator{\rw}{rw}
\DeclareMathOperator{\NAT}{NAT}
\DeclareMathOperator{\std}{std}
\DeclareMathOperator{\South}{South}
\DeclareMathOperator{\dg}{dg}
\DeclareMathOperator{\inc}{inc}
\DeclareMathOperator{\dec}{dec}
\DeclareMathOperator{\ID}{ID}

\DeclareMathOperator{\Nu}{\widehat{V}}

\newlength\cellsize 
\savebox2{%
\begin{picture}(12,12)
\put(0,0){\line(1,0){12}}
\put(0,0){\line(0,1){12}}
\put(12,0){\line(0,1){12}}
\put(0,12){\line(1,0){12}}
\end{picture}}
\newcommand\cellify[1]{\def\thearg{#1}\def\nothing{}%
\ifx\thearg\nothing
\vrule width0pt height\cellsize depth0pt\else
\hbox to 0pt{\usebox2\hss}\fi%
\vbox to 12\unitlength{
\vss
\hbox to 12\unitlength{\hss$#1$\hss}
\vss}}
\newcommand\tableau[1]{\vtop{\let\\=\cr
\setlength\baselineskip{-16000pt}
\setlength\lineskiplimit{16000pt}
\setlength\lineskip{0pt}
\halign*{&\cellify{##}\cr#1\crcr}}}
\savebox3{%
\begin{picture}(12,12)
\put(0,0){\line(1,0){12}}
\put(0,0){\line(0,1){12}}
\put(12,0){\line(0,1){12}}
\put(0,12){\line(1,0){12}}
\end{picture}}
\newcommand\expath[1]{%
\hbox to 0pt{\usebox3\hss}%
\vbox to 12\unitlength{
\vss
\hbox to 12\unitlength{\hss$#1$\hss}
\vss}}

\newcommand\cell[3]{
\def\i{#1} \def\j{#2} \def\entry{#3}
\draw (\j-1,-\i)--(\j,-\i)--(\j,-\i+1)--(\j-1,-\i+1)--(\j-1,-\i);
\node at (\j-.5,-\i+.5) {\entry};
}

\title{Expanding the quasisymmetric Macdonald polynomials in the fundamental basis}

\date{\today}
\author[Corteel]{Sylvie Corteel}
\address{Department of Mathematics, UC Berkeley, USA}
\email{corteel@berkeley.edu}
\author[Mandelshtam]{Olya Mandelshtam}
\address{Department of Mathematics, Brown University, USA}
\email{olya@math.brown.edu}
\author[Roberts]{Austin Roberts}
\address{Department of Mathematics, Highline College, USA}
\email{aroberts@highline.edu}

\begin{document}

\maketitle

\begin{abstract}
The quasisymmetic Macdonald polynomials $G_{\gamma}(X;q,t)$ were recently introduced by the first and second authors with Haglund, Mason, and Williams in \cite{CHMMW} to refine the symmetric Macdonald polynomials $P_{\lambda}(X;q,t)$ with the property that $G_{\gamma}(X;0,0)$ equals $QS_{\gamma}(X)$, the quasisymmetric Schur polynomial of \cite{HLMvW11}. We derive an expansion for $G_{\gamma}(X;q,t)$ in the fundamental basis of quasisymmetric functions.
\end{abstract}

\section{Introduction}

The symmetric \emph{Macdonald polynomials} $P_{\lambda}(X; q, t)$ \cite{Macdonald} are a family of functions in $X = \{x_1, x_2,\dots \}$ indexed by partitions, whose coefficients depend on two parameters $q$ and $t$. They can be defined as the unique monic basis for the ring of symmetric functions that satisfies certain triangularity and orthogonality conditions. Macdonald polynomials generalize multiple important families of polynomials, including Schur polynomials and Hall--Littlewood polynomials.
 
The related \emph{nonsymmetric Macdonald polynomials} $E_{\mu}(X;q,t)$ were introduced shortly after as a tool to study Macdonald polynomials, in a series of papers by Cherednik \cite{Cher1}, Macdonald \cite{Mac95}, and Opdam \cite{Opd95}.
The polynomials $E_{\mu}(X;q, t)$ are indexed by weak compositions and form a basis for the full polynomial ring $\mathbb{Q}[X](q,t)$. Ferreira \cite{Fer11} and later Alexandersson \cite{Ale16} studied the extension of these to the more general \emph{permuted basement} nonsymmetric Macdonald polynomials $E^\sigma_{\mu}(X;q,t)$, where $X= \{x_1, \ldots, x_n\}$,  $\sigma\in S_n$, and the length of $\mu$ is $n$.

The combinatorics of Macdonald polynomials has been actively studied for decades. In \cite{HHL05}, Haglund, Haiman, and Loehr gave a combinatorial formula for the \emph{modified Macdonald polynomials},  $\widetilde{H}_{\lambda}(X;q,t)$. A formula for the \emph{integral form},  $J_{\lambda}(X; q,t)$, was then given in \cite{HHL08}. 
Important to our purposes, this paper also provided a formula for the nonsymmetric Macdonald polynomials $E_{\mu}(X;q,t)$, which was then broadened to the more general polynomials  $E^\sigma_{\mu}(X;q,t)$ in \cite{Ale16,Fer11}.

In \cite{CHMMW}, the first and second authors together with Haglund, Mason, and Williams %
introduced a new family of quasisymmetric functions $G_{\gamma}(X; q, t)$ they named \emph{quasisymmetric Macdonald polynomials}.  They showed that $G_{\gamma}(X; q, t)$ is indeed a quasisymmetric function, and gave a combinatorial formula for the $G_{\gamma}(X; q, t)$ %
that refines the compact formula for the $P_{\lambda}$ from \cite{CMW18}. The Macdonald polynomial $P_{\lambda}(X; q, t)$ is a sum of these quasisymmetric Macdonald polynomials, and %
at $q=t=0$, $G_{\gamma}(X; q, t)$ specializes to the \emph{quasisymmetric Schur functions} $\text{QS}_{\gamma}(X)$ introduced by Haglund, Luoto, Mason, and van Willigenburg in \cite{HLMvW11}.  

The goal of this article is to write an expansion of the polynomials $G_{\gamma}(X; q, t)$ in the fundamental basis. This basis was introduced by Gessel in \cite{Ges84} and is one of the most common bases of the vector space of quasisymmetric functions. Our main results are the following Theorems, see \cref{sec:def} for the relevant definitions.
\begin{theorem}\label{thm:q}
Let $\gamma$ be a strong composition. Then
\begin{align*}
G_{\gamma}(X;q,t)= & \sum_{\tau\in\ST(\gamma)} t^{\coinv(\tau)}q^{\maj(\tau)} \left( \prod_{\substack{u\in\widehat{\dg}(\gamma)\\u\not\in W(\tau)}} \frac{1-t}{1-q^{\leg(u)+1}t^{\arm(u)+1}} \right) \\
&\times \sum_{U \subseteq W(\tau)} (-t)^{|U|} \left( \prod_{u \in U} \frac{1-q^{\leg(u)+1}t^{\arm(u)}}{1-q^{\leg(u)+1}t^{\arm(u)+1}} \right)F_{V(\tau)\cup U}.
\end{align*}\label{main}
\end{theorem}

\begin{theorem}\label{thm:HL}
Let $\gamma$ be a strong composition. Then
\begin{align*}
G_{\gamma}(X;0,t)& =
\sum_{\tau\in\ST_{1}(\gamma)}
(1-t)^{\omega(\tau)}
(-t)^{|\Des(\tau)|}
t^{\coinv(\tau)-\coinv(\Des(\tau))} 
 F_{\Nu(\tau)}.
\end{align*}

\end{theorem}

This article proceeds through a series of purely combinatorial proofs and results using a variety of tableaux enumeration techniques, organized as follows.
In \cref{sec:def}, we provide the relevant background. \cref{sec:packed} provides a proof for \cref{thm:q}. In \cref{sec:HL} we provide an alternative expansion in the Hall-Littlewood case, yielding \cref{thm:HL} and a related result for Jack polynomials.

\section{Preliminaries and definitions}
\label{sec:def}

For a nonnegative integer $n$, a \emph{weak composition} $\alpha=(\alpha_1,\ldots,\alpha_k) \models n$ is a list of nonnegative integers called the \emph{parts} of $\alpha$, summing to $n$, so that $n=|\alpha|=\sum_{i=1}^k \alpha_i$.  Let $\alpha^+$ denote the composition obtained by collapsing the (weak) composition $\alpha$ by removing the zero-parts from $\alpha$. We call a composition with no non-zero parts a \emph{strong composition}. If $\alpha_1\geq \alpha_2 \geq \cdots \geq \alpha_k$, then $\alpha$ is called a \emph{partition}. We denote by %
$\inc(\alpha)$ the composition obtained by sorting the parts of $\alpha$ in increasing order. Define $\beta(\alpha)$ to be the permutation of \emph{longest length} such that $\beta(\alpha) \circ \alpha = \inc(\alpha)$, where the length of a permutation is the number of inversions in its word representation.

\begin{example}
For $\alpha=(2,1,0,0,3,0,1)$, we have $\alpha^+=(2,1,3,1)$, %
$\inc(\alpha)=(0,0,0,1,1,2,3)$, and $\beta(\alpha)=(6,4,3,7,2,1,5)$.
\end{example}

For two compositions, $\alpha,\beta$, we say $\beta$ is a \emph{refinement} of $\alpha$, denoted by $\beta<\alpha$, if $\alpha$ can be obtained by adding together adjacent parts of $\beta$. For example, $(2,1,3,1)<(2,4,1)<(2,5)<(7)$. Finally, there is a natural bijection from (strong) compositions $\alpha\models n$ with $|S|+1$ parts to subsets $S\in [n-1]$, given by taking the difference between successive elements of $\{0,n\}\cup{S}$, after elements of this set are listed in order. Specifically, the subset $S$ corresponding to a composition $\alpha=(\alpha_1,\ldots,\alpha_k)$ is
\[
S = \{\alpha_1,\alpha_1+\alpha_2,\ldots,\alpha_1+\alpha_2+\cdots+\alpha_{k-1}\},
\]
and the composition $\alpha\models n$ corresponding to a subset $S=\{i_1,i_2,\ldots,i_{k-1}\}\subseteq [n-1]$ is
\[
\alpha=(i_1,i_2-i_1,i_3-i_2,\ldots,n-i_{k-1}).
\]

\begin{example}
$\alpha=(2,1,3,2)$ corresponds to the subset $S=\{2,3,6\}\subseteq [7] $. 
\end{example}

\subsection{Quasisymmetric functions}
The vector space of \emph{quasisymmetric functions} properly contains the space of symmetric functions. A quasisymmetric function is a bounded degree formal power series $f\in \mathbb{F}[x_1,x_2,\ldots]$ such that for all $k$, all compositions $\alpha=(\alpha_1,\alpha_2,\ldots,\alpha_k)$, and all sets of indices $i_1<i_2<\cdots<i_k$, the coefficients of $x_1^{\alpha_1}\ldots x_k^{\alpha_k}$ and $x_{i_1}^{\alpha_1}\ldots x_{i_k}^{\alpha_k}$ in $f$ are equal. 

Similar to the symmetric functions, the vector space of quasisymmetric functions has several natural bases consisting of functions of fixed degree. We will focus on the \emph{monomial basis} $\{M_S\}$ and the \emph{fundamental basis} $\{F_S\}$, indexed by subsets $S \subset [n-1]$, for each fixed degree $n$. The monomial basis functions are defined as
\begin{equation}
M_S := \sum_{i_1 < i_2 < \cdots < i_{k}} x_{i_1}^{\alpha_1}x_{i_2}^{\alpha_2}\cdots x_{i_k}^{\alpha_k}
\end{equation}
where $k=|S|+1$, and $\alpha$ is the (strong) composition corresponding to the subset $S$.

The fundamental basis functions are defined as
\begin{equation}
F_S := \sum_{i_1 \leq i_2 \leq \cdots \leq i_n \atop  j \in S \implies i_j \neq i_{j+1}} x_{i_1}x_{i_2}\cdots x_{i_n}.
\end{equation}

\noindent
For example, 
\begin{align*}
M_{\{2,3,6\}} &= \sum_{i_1 < i_2 < i_3 < i_{4}} x_{i_1}^{2}x_{i_2}^{1}x_{i_3}^3 x_{i_4}^{2},\\
F_{\{2,3,6\}}&=\sum_{i_1 \leq i_2 < i_3 < i_4\leq i_5\leq i_6<i_7\leq i_8} x_{i_1}x_{i_2}\cdots x_{i_8}.
\end{align*}

Let $S\subseteq [n-1]$. It follows that %
\begin{equation}\label{eq:F}
F_{S}=\sum_{S\subseteq S'} M_{S'}.
\end{equation}

\noindent
For example, let $n=8$ and $S=\{1,4\}$. Then %
\begin{align*}
F_{\{1,4\}} &= M_{\{1,4\}}+M_{\{1,2,4\}}+M_{\{1,3,4\}}+M_{\{1,2,3,4\}}.
\end{align*}

The goal of this article is to give an expansion of the quasisymmetric Macdonald polynomial $G_{\gamma}(X;q,t)$, which we present below, in terms of the fundamental quasisymmetric basis. Let $\gamma$ be a strong composition. The quasisymmetric Macdonald polynomial is defined by the infinite sum
\begin{align}\label{def: Mac}
G_{\gamma}(X;q,t) =  \sum_{\alpha:\ \alpha ^{+}= \gamma}  E_{\inc(\alpha)}^{\beta (\alpha )}(X;q,t),
\end{align}
where $E_{\mu}^{\sigma}(X;q,t)$ is the \emph{permuted basement Macdonald polynomial} introduced in \cite{Fer11} and further studied in \cite{Ale16}. We will define $G_{\gamma}$ combinatorially in the next section. Note that $E_{\mu}^{\sigma}$ is a polynomial in $k$ variables, where $k$ is the number of parts of $\mu$, so we actually mean $E_{\mu}^{\sigma}(X;q,t)=E_{\mu}^{\sigma}(x_1,\ldots,x_k;q,t)$, and $\sigma\in S_k$. %

\begin{remark}
Due to \cite{CHMMW}, it turns out that $E_{\inc(\alpha)}^{\beta (\alpha )}(X;0,t) = E_{\alpha}^{id}(X;0,t)$. Thus the quasisymmetric Hall--Littlewood polynomials $\mathcal{L}_{\alpha}(X;t)$, defined in \cite{HLMvW11} as
\[
\mathcal{L}_{\gamma}(X;t) = \sum_{\alpha : \alpha^+ = \gamma} E_{\alpha}^{id}(X;0,t)
\]
coincide with $G_{\gamma}(X;0,t)$.
\end{remark}

\subsection{Tableaux formula for $E_{\mu}^{\sigma}(X;q,t)$}

The polynomial $E_{\mu}^{\sigma}(X;q,t)$ %
has a combinatorial description in the form of a tableaux formula \cite{HHL05}. We review the relevant statistics for general compositions, though we  will be primarily concerned with the case where the parts of $\mu$ are arranged in weakly increasing order.

For any weak composition $\alpha$, define $\dg(\alpha)$, the diagram of $\alpha$, to be the composition shape in French notation with $\alpha_i$ boxes in column $i$ from left to right. The rows are labeled from bottom to top starting with row 1, and a cell in row $r$ and column $c$ is denoted by coordinates $(r,c)\in\dg(\alpha)$. Define $\widehat{\dg}(\alpha)$ to be the set of cells in $\dg(\alpha)$ that are not contained in the bottom row.  If $T$ is a filling of $\dg(\alpha)$, the entry in a cell $u\in\dg(\alpha)$ is denoted by $T(u)$. Let $x^T=\prod_{u \in \dg(\alpha)} x_{T(u)}$, be the monomial encoding the content of $T$.

The \emph{reading order} of a diagram is the total order given by reading the entries along the rows from top to bottom, and from left to right within each row. Two cells are said to \emph{attack} each other if they are in the same row, or if they are in adjacent rows where the one above is strictly northeast of the one below. %
A filling $T$ is considered \emph{non-attacking} if $T(u)\neq T(v)$ for any pair of attacking cells $u,v$. %

For a cell $u\in\dg(\alpha)$, we call $\leg(u)$ the number of cells above $u$ in the same column. We call $\arm(u)$ the number of cells to the right of $u$ in columns whose height does not exceed the height of the column containing $u$, plus the number of cells to the left of $u$ in columns of height strictly smaller than the height of the column containing $u$. More precisely, let $u=(r,i)$. Then
\begin{align*}
\arm(u) =& |\{(r,j)\in\dg(\alpha):\ j>i, \alpha_j\leq \alpha_i\}|+|\{(r-1,j)\in\dg(\alpha): j<i, \alpha_j< \alpha_i\}|
\end{align*} 
See \cref{fig:armleg}. Denote by $\South(u)$ the cell directly below $u$ in the same column. The set of descents of a filling of $\dg(\alpha)$ is
\[
\Des(T) = \{u\in\widehat{\dg}(\alpha)\ :\ T(u)>T(\South(u))\},
\]
and the \emph{major index} is
\[
\maj(T) = \sum_{u\in\Des(T)} \leg(u)+1.
\]

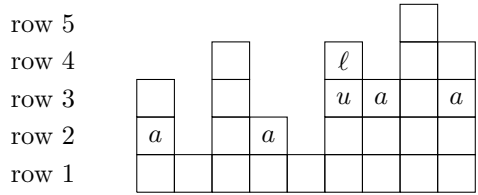
\begin{figure}
  \centering
\begin{tikzpicture}[scale=.5]
\cell39{$a$} \cell38{} \cell37{$a$} \cell36{$u$}  \cell33{$$}\cell31{$$}
\cell49{} \cell48{} \cell47{$$}\cell46{}  \cell44{$a$} \cell43{} \cell41{$a$} %
\cell18{} 
\cell29{} \cell28{} \cell26{$\ell$} \cell23{}
\cell59{$$} \cell58{$$}  \cell57{$$} \cell56{$$}  \cell55{$$}  \cell54{$$}  \cell53{$$}  \cell52{$$}  \cell51{$$} 

\node at (-2.5,-.5) {row 5};
\node at (-2.5,-1.5) {row 4};
\node at (-2.5,-2.5) {row 3};
\node at (-2.5,-3.5) {row 2};
\node at (-2.5,-4.5) {row 1};
\end{tikzpicture} 
\caption{The diagram of the composition $(3,1,4,2,1,4,3,5,4)$ and the cells in the leg and the arm of the cell $u=(3,6)$. Here $\leg(u)=1$ and $\arm(u)=4$.}
\label{fig:armleg}
\end{figure}

\emph{Triples} consist of a cell $x$, the cell $y=\South(x)$ directly below, and a third cell $z$ in the arm of $x$. If $z$ is in the same row as $x$, this is called a \emph{type A triple}, and if $z$ is in the same row as $y$, this is called a \emph{type B triple}, as shown

\begin{align*}
\begin{tabular}{b{3.5em}m{5em}b{3.5em}m{3.5em}}
\text{Type A: } & \begin{tikzpicture}[scale=0.4]\cell10{$x$}\cell20{$y$} \cell12{$z$}\end{tikzpicture}&
\text{ Type B: }& \begin{tikzpicture}[scale=0.4]\cell20{$z$} \cell12{$x$}\cell22{$y$}\end{tikzpicture}
\end{tabular}
\end{align*}

Coinversion triples consist of type A triples where the entries are increasing in clockwise orientation, plus type B triples where the entries are increasing in counterclockwise orientation. The $\coinv(T)$ statistic is defined as the total number of all such triples. 

Let $\gamma$ be a strong composition, and let $\sigma=\beta(\gamma)$ be the permutation of longest length such that $\sigma \circ \gamma=\inc(\gamma)$. Define $\NAT(\gamma)$ to be the set of non-attacking fillings of $\dg(\inc(\gamma))$ such that the entries of the first row are order-equivalent to $\sigma$ when read in reading order. %

\begin{example}
Let $\alpha=(0,4,0,3,1,0,0,3)$. Then $\inc(\alpha)=(0,0,0,0,1,3,3,4)$,  $\alpha^+=(4,3,1,3)$, and $\beta(\alpha)=(7,6,3,1,5,8,4,2)$. The NAT associated to $\alpha$ are fillings of $\dg(\inc(\alpha^+))$ with the bottom row equal to $(5,8,4,2)$: the last $\ell$ entries of $\beta(\alpha)$, where $\ell=\ell(\alpha^+)=4$. %
Notice that $(5,8,4,2)$, is order-equivalent to $(3,4,2,1)$, and $(3,4,2,1)=\beta(\alpha^+)$. Thus in particular, all tableaux associated to $\alpha$ also belong to $\NAT(\alpha^+)$, such as the one below.
\begin{center}
\begin{tikzpicture}[scale=.5]
\cell13{5}
\cell21{2} \cell22{$7$} \cell23{$5$}  
\cell31{$4$}\cell32{$1$}\cell33{2} 
\cell40{$5$} \cell41{$8$}\cell42{$4$}  \cell43{$2$} 
\node at (6.5,-2) {$\in \NAT((4,3,1,3))$};
\end{tikzpicture} %
\end{center}
\end{example}

By comparing with \cite{Ale16}, we obtain the combinatorial formula for $E_{\inc(\alpha)}^{\beta(\alpha)}(X;q,t)$, where $\alpha$ is a weak composition:
\begin{equation}\label{eq:E}
E_{\inc(\alpha)}^{\beta (\alpha)} (X; q,t) = \sum_{\substack{T\in\NAT(\alpha^+)\\T\ \text{has bottom row } \pi}} %
\wt(T)x^T,
\end{equation}
where $\pi$ is the last $\ell$ entries of $\beta(\alpha)$, for $\ell=\ell(\alpha^+)$.
Here, the weight of a (nonstandard) filling $T$ is
\begin{align}
\wt(T)= q^{\maj(T)}t^{\coinv(T)} \prod_{\substack{u \in \widehat{\dg}(\alpha^+)\\T(u) \neq T(\South(u))}} \frac{(1-t)}{(1-q^{\leg(u)+1}t^{\arm(u)+1})}
\end{align}

\begin{remark}
We have given the tableaux formula for $E_{\mu}^{\sigma}$ where the parts of $\mu$ are weakly increasing. A general formula exists (see \cite{Ale16} for details) for an arbitrary composition $\mu$ and a permutation $\sigma$ by keeping track of the ``basement'' of a filling. Comparing definitions, it follows that for any composition $\alpha$, the basement of a filling of $\dg(\inc(\alpha))$ can be recovered uniquely from the bottom row of the filling.
\end{remark}

\subsection{Standard, packed, and non attacking fillings}

A \emph{packed} filling is one that uses every integer from the set $\{1,\ldots,m\}$ for some $m$. %
For example, the filling \begin{tikzpicture}[scale=0.4]\cell01{$2$} \cell11{$2$}\cell10{$1$}\end{tikzpicture} is packed, but \begin{tikzpicture}[scale=0.4]\cell01{$3$} \cell11{$3$}\cell10{$1$}\end{tikzpicture} is not. However, the latter filling compresses to the former by shifting the alphabet of values in the filling down as necessary: given a set $\{s_1,\ldots,s_k\}$ with $s_1<\cdots<s_k$, the entries $s_i$ become $i$.

It is convenient to work with packed fillings in the context of quasisymmetric functions. We consider every packed filling $T$ to be the representative of the family of fillings which compress to $T$. 
\begin{lemma}
Suppose $T'\in\NAT(\gamma)$ compresses to a packed filling $T\in\NAT(\gamma)$. Then $\coinv(T')=\coinv(T)$ and $\maj(T')=\maj(T)$.
\end{lemma}

The proof of the above lemma follows from the fact that the relative order of entries is preserved by compression. Moreover,
\[
\sum_{T'} x^{T'} = M_T,
\]
the sum being over all fillings $T'$ that compress to the packed filling $T$, and $M_T$ is the monomial quasisymmetric function corresponding to the content of $T$. Thus the $q,t$-generating function of the family of fillings that compress to the packed representative $T$ is the weight of $T$ times $M_T$. 
Hence, we may work with the finite set of packed fillings to represent all possible fillings.

From \eqref{def: Mac} and \eqref{eq:F}, we thus obtain
\begin{equation}\label{eq:G}
G_{\gamma}(X;q,t) %
= \sum_{\substack{T\in\NAT(\gamma)\\T\ \text{packed}}} q^{\maj(T)}t^{\coinv(T)}M_T \prod_{\substack{u \in \widehat{\dg}(\gamma)\\T(u) \neq T(\South(u))}} \frac{(1-t)}{(1-q^{\leg(u)+1}t^{\arm(u)+1})}.
\end{equation}

\begin{example}
For $\gamma=(1,2)$, all the packed nonattacking fillings in $\NAT(\gamma)$ are shown below with their weights, to obtain
\[
G_{(1,2)} = M_{\{1\}}+\frac{(1-t)(1+t+qt)}{1-qt^2}M_{\{1,1\}}.
\]
\begin{center}
\begin{tikzpicture}[scale=.5]
\node at (-4,0) {$\gamma=(1,2)$:};
\cell01{2}
\cell10{1} \cell11{2}
\node at (0,-2) {$M_{\{1\}}$};
\cell06{3}
\cell15{1} \cell16{2}
\node at (5,-2) {$\frac{qt(1-t)}{1-qt^2}M_{\{1,1\}}$};
\cell0{11}{2}
\cell1{10}{1} \cell1{11}{3}
\node at (10,-2) {$\frac{(1-t)}{1-qt^2}M_{\{1,1\}}$};
\cell0{16}{1}
\cell1{15}{2} \cell1{16}{3}
\node at (15,-2) {$\frac{t(1-t)}{1-qt^2}M_{\{1,1\}}$};
\end{tikzpicture}
\end{center}
\end{example}

\emph{Standard fillings} (or standard tableaux), denoted by $\ST(\gamma)$, are fillings of $\dg(\inc(\gamma))$ such that every element in the set $\{1,\ldots,n\}$ appears exactly once, where $n=|\gamma|$. Thus there is a bijection $\tau:\dg(\inc(\gamma))\rightarrow\{1,\ldots,n\}$ between cells of $\dg(\inc(\gamma))$ and the entries $\{1,\ldots,n\}$, and so we can slightly abuse notation and refer to both a cell and its entry when we work with standard tableaux. %

The standardization map $\std \colon \NAT(\gamma) \rightarrow \ST(\gamma)$ is defined as follows. For $T\in\NAT(\gamma)$, let $\tau=\std(T)$ be the unique standard filling in $\ST(\gamma)$ that preserves the relative order of the original tableau, and where the reading order is used to break ties. In other words, $\tau$ is such that the sequence $T\circ \tau^{-1}$ is weakly increasing, and the restriction of $\tau$ to $T^{-1}(\{x\})$ is increasing with respect to the reading order for all values $x$. It is straightforward to check that if $\tau$ is the standardization of $T$, then $\coinv(T)=\coinv(\tau)$ and $\maj(T)=\maj(\tau)$. See \cref{example:std} below for the standardization $\std(T)$ of $T\in\NAT((1,4,3))$. 

Let $T\in\NAT(\gamma)$ with standardization $\tau=\std(T)$, and $n=|\gamma|$. Define the \emph{reading word} of $T$ to be the sequence of entries of $T$ listed in reading order, denoted by $\rw(T)$. The reading word of $\tau$ is thus a permutation of $\{1,\ldots,n\}$. Define $\ID(\tau)$ to be the \emph{inverse descent set}, where  $i\in \ID(\tau)$ if $i+1$ precedes $i$ in $\rw(\tau)$.
 Comparing definitions, one may check that $\tau$ is a standardization of $T$ if and only if the following holds: the weakly increasing sequence $a=T\circ\tau^{-1}$ satisfies that if $a(i)=a(i+1)$, then $i\not\in \ID(\tau)$. See \cref{example:std} for an explicit computation.
\begin{example}\label{example:std}
We show $T\in\NAT((1,4,3))$ and the corresponding standardization $\tau=\std(T)$.
\begin{center}
\begin{tikzpicture}[scale=0.5]
\node at (-2,-1) {$T=$};
\cell022
\cell121\cell113
\cell223\cell215
\cell322\cell314\cell301
\end{tikzpicture} 
\qquad
\begin{tikzpicture}[scale=0.5]
\node at (-3,-1) {$\tau=\std(T)=$};
\cell023
\cell121\cell115
\cell226\cell218
\cell324\cell317\cell302
\end{tikzpicture} 
\end{center}
We have $a=T\circ\tau^{-1}=(1,1,2,2,3,3,4,5)$, $\rw(\tau)=(3,5,1,8,6,2,7,4)$, and $\ID(\tau)=\{2,4,7\}$. We check that the indices $i$ for which $a(i)=a(i+1)$ are $\{1,3,5\}$, and indeed those are not contained in $\ID(\tau)$.
\end{example}

For $\tau\in\ST(\gamma)$, define %
 $V(\tau)\subseteq [n-1]$ to be the set of entries such that $i\in V(\tau)$ if $i\in\ID(\tau)$ %
or if $i$ and $i+1$ are in cells that attack each other. Define $W(\tau) = \{i\in \tau\ :\ \South(i)=i+1\}$ to be the set of entries $i$ with $i+1$ directly below. Note that $V(\tau) \cap W(\tau) = \emptyset$.

Given a standard filling $\tau$ with $n$ cells, the cells labelled from 1 to $n-1$ are partitioned into three blocks:
\begin{itemize}
\item The cells containing entries in $V(\tau)$, namely those cells where $i\in\ID(\tau)$ %
{\bf OR} $i$ and $i+1$ are in attacking cells.
\item The cells containing entries in $W(\tau)$, namely those cells where $i+1$ is directly below $i$.
\item The rest of the cells with entries in $[n-1]\backslash (V(\tau)\cup W(\tau))$.
\end{itemize}

Let $\gamma \models n$. We consider the pre-image in $\NAT(\gamma)$ of standard fillings $\tau\in\ST(\gamma)$. For a choice of $V(\tau)\subseteq S\subseteq [n-1]$, define a \emph{destandardization map} $\delta_S(\tau)\ :\  \dg(\gamma) \rightarrow \mathbb{Z}$ as follows. Let $\alpha$ be the (strong) composition corresponding to the set $S$. Let $w$ be the word containing the content associated to $\alpha$ in weakly decreasing order, given by $w=(1^{\alpha_1},2^{\alpha_2},\ldots,k^{\alpha_k})$ where $\alpha$ has $k$ parts. Define $\delta_S(\tau):= w\circ \tau$ to be the unique filling of $\dg(\gamma)$ with content $\alpha$ that standardizes to $\tau$. %

\begin{example}
Consider the standard tableau $\tau$ from \cref{example:std}. $\ID(\tau)=\{2,4,7\}$, and the set of indices $i$ such that $i$ and $i+1$ are in cells that attack each other is $\{6\}$, so $V(\tau)=\{2,4,6,7\}$. Thus, $S$ can be any subset of $[7]$ containing $V(\tau)$. We show some examples of $\delta_S$ for various choices of $S$:
\begin{center}
\begin{tikzpicture}[scale=0.5]
\node at (-3,-1) {$\delta_{\left\{\substack{1,2,3,\\4,5,6,7}\right\}}(\tau)=$};
\cell023
\cell121\cell115
\cell226\cell218
\cell324\cell317\cell302
\end{tikzpicture} 
\quad 
\begin{tikzpicture}[scale=0.5]
\node at (-2.9,-1) {$\delta_{\left\{\substack{1,2,4,\\5,6, 7}\right\}}(\tau)=$};
\cell023
\cell121\cell114
\cell225\cell217
\cell323\cell316\cell302
\end{tikzpicture} 
\quad 
\begin{tikzpicture}[scale=0.5]
\node at (-2.8,-1) {$\delta_{\left\{ \substack{2,4, \\ 6, 7}\right\}}  (\tau)=$};
\cell022
\cell121\cell113
\cell223\cell215
\cell322\cell314\cell301
\end{tikzpicture} 
\end{center}

\end{example}

The following Lemmas contain some observations that will be essential to our arguments.

\begin{lemma}\label{lem:std}
Let $\gamma\models n$, and let $\tau\in\ST(\gamma)$. For a set $S\subseteq [n-1]$ such that $V(\tau)\subseteq S$, the following conditions hold: 
\begin{enumerate}[i.]
\item $\delta_S(\tau)\in \NAT(\gamma)$.
\item $\std(\delta_S(\tau))=\tau$.
\item Let $S$ be the unique set such that $T=\delta_S(\tau)$. Then $\wt(T) =  \wt(\delta_{S}(\std(T)))$.
\end{enumerate}
\end{lemma}

In fact, the destandardization map splits $\NAT(\gamma)$ into disjoint components $\{T\in \NAT(\gamma)\ :\ \std(T)=\tau\}$ indexed by $\tau\in\ST(\gamma)$, due to the following lemma.

\begin{lemma}\label{lem:delta}
Let $\gamma\models n$, let $T\in\NAT(\gamma)$ be a packed filling, and set $\tau=\std(T)$. There exists a unique set $V(\tau)\subseteq S \subseteq [n-1]$ such that $T=\delta_S(\tau)$. 
Furthermore,
\begin{align*}
\sum_{T\ :\ \std(T) =\tau} x^T 
=\sum_{V(\tau) \subseteq S \subseteq [n-1]} \sum_{T=\delta_S(\tau)} M_T
= \sum_{V(\tau) \subseteq S \subseteq [n-1]} M_{S}.
\end{align*} 
\end{lemma}

Thus for a strong composition $\gamma$, the set of packed fillings contained in $\NAT(\gamma)$ is
\[
\bigsqcup_{\tau\in\ST(\gamma)} \bigcup_{V(\tau)\subseteq S\subseteq [n-1]} \delta_S(\tau).
\]

The following lemma gives the weight of a destandardized filling in terms of its standardization.

\begin{lemma}\label{lem:wtq}
Let $\gamma\models n$, $\tau\in\ST(\gamma)$, and $S$ such that $V(\tau)\subseteq S\subseteq [n-1]$. Then 
\begin{align*}
\wt(\delta_S(\tau)) = &\ t^{\coinv(\tau)} q^{\maj(\tau)} \prod_{\substack{u \in \widehat{\dg}(\gamma)\\ u\not\in W(\tau)}} \frac{1-t}{1-q^{\leg(u)+1}t^{\arm(u)+1}}
 \\
& \times \prod_{u \in S\cap W(\tau)} \frac{1-t}{1-q^{\leg(u)+1}t^{\arm(u)+1}}.
\end{align*}
\end{lemma}
\begin{proof}
We are given $V(\tau)\subseteq S\subseteq [n-1]$, and let $T=\delta_S(\tau)$. Recall that $W(\tau)$ contains the entries $\{i\in\tau: \South(i)=i+1\}$.  Ensuring the products are correct is a matter of tracking the cells $\{u\in\widehat{\dg}(T) : T(\South(u))=T(u)\}$.  All such cells must correspond to cells of $W(\tau)$.  %
In particular, by our construction, $\{u\in\widehat{\dg}(T):T(u)=T(\South(u))\} = W(\tau)\backslash S$.  All other cells contribute a factor of $\frac{1-t}{1-q^{\leg(u)}t^{\arm(u)}}$.
Since the coinversions and major index are preserved by standardization, the lemma follows.
\end{proof}

For a strong composition $\gamma \models n$ where $\ell(\gamma)$ is the number of parts, define $h(\gamma)=n-\ell(\gamma)$ to be the number of cells in $\dg(\gamma)$ without its bottom row. Note that $h(\gamma)$ is the number of cells in $\widehat{\dg}(\gamma)$. 

\section{Proof of \cref{thm:q}}\label{sec:packed}

We will start with a proof for the $q=0$ specialization of \cref{thm:q} to develop the main ideas of the proof. The proof for the general $q$ case follows the same strategy.

\subsection{The $q=0$ case}\label{sec:q0}

We assume $q=0$ throughout this section. Let $\gamma$ be a strong composition. When we compute compute $G_{\gamma}(X;0,t)$, the only surviving tableaux in \eqref{eq:G} are those with an empty descent set, which means the entries must be non-increasing as we read the columns from bottom to top. We denote the subsets of $\NAT(\gamma)$ and $\ST(\gamma)$ that have nonzero weight at $q=0$ by $\NAT_0(\gamma)$ and $\ST_0(\gamma)$, respectively.  We will prove the following.
\begin{equation}\label{eq:HL}
G_{\gamma}(X;0,t) = \sum_{\tau\in\ST_0(\gamma)} t^{\coinv(\tau)}(1-t)^{h(\gamma)-|W(\tau)|} \sum_{U\subseteq W(\tau)} (-t)^{|U|} F_{V(\tau)\cup U}.
\end{equation}

Observe the following. The denominator in the product of \eqref{eq:G} vanishes, and the weight of each $T\in\NAT_0(\gamma)$ becomes
\[
\wt(T)=t^{\coinv(T)}(1-t)^{|\{u \in \widehat{\dg}(\lambda)\ :\ T(\South(u))\neq T(u)\}|}.
\]
Moreover, for $\tau\in\ST_0(\gamma)$, \cref{lem:wtq} specializes to
\begin{equation}\label{eq:wt}
\wt(\delta_S(\tau))=t^{\coinv(\tau)}(1-t)^{h(\gamma)-|W(\tau)\backslash S|}.
\end{equation}

We are now ready to prove \cref{thm:q} at $q=0$. %

\begin{proof}[Proof of \eqref{eq:HL}]
From the definition, we have
\begin{align}
G_{\gamma}(X;0,t) &= \sum_{T\in\NAT_0(\gamma)} \wt(T)x^T \nonumber\\
&=\sum_{\tau\in\ST_0(\gamma)} \sum_{V(\tau)\subseteq S \subseteq [n-1]} \wt(\delta_S(\tau))M_{S}\nonumber\\
&=\sum_{\tau\in\ST_0(\gamma)} \sum_{V(\tau)\subseteq S \subseteq [n-1]} t^{\coinv(\tau)}(1-t)^{h(\gamma)-|W(\tau)\backslash S|} M_{S}\nonumber\\
&=\sum_{\tau\in\ST_0(\gamma)}   t^{\coinv(\tau)}(1-t)^{h(\gamma)-|W(\tau)|}\sum_{V(\tau)\subseteq S \subseteq [n-1]}  
(1-t)^{|S\cap W(\tau)|}M_{S} \label{eq:G2}
\end{align}
where the second line is by \cref{lem:std,lem:delta}, the third line is by \eqref{eq:wt}.%

To complete the proof, we need to reformulate the second summation.
\begin{align*}\sum_{V(\tau)\subseteq S \subseteq [n-1]}  
(1-t)^{|S\cap W(\tau)|}M_{S} 
&=\sum_{W\subseteq W(\tau)} (1-t)^{|W|} \sum_{\substack{V(\tau)\subseteq S \subseteq [n-1]\\S\cap W(\tau)=W}} M_{S}
\end{align*}
By the binomial theorem,
\begin{equation}\label{eq:W}
(1-t)^{|W|} = \sum_{U \subseteq W} (-t)^{|U|}.
\end{equation}
Plugging in gives
\begin{align*}
\sum_{W\subseteq W(\tau)}\sum_{U \subseteq W} (-t)^{|U|} \sum_{\substack{V(\tau)\subseteq S \subseteq [n-1]\\S\cap W(\tau)=W}} M_{S}
&=\sum_{U \subseteq W(\tau)} (-t)^{|U|} \sum_{W\supseteq U}\sum_{\substack{V(\tau)\subseteq S \subseteq [n-1]\\S\cap W(\tau)=W}} M_{S}\\
&=\sum_{U \subseteq W(\tau)} (-t)^{|U|} \sum_{V(\tau)\cup U\subseteq S \subseteq [n-1]} M_{S}\\
&=\sum_{U \subseteq W(\tau)} (-t)^{|U|} F_{V(\tau)\cup U},
\end{align*}
which completes the proof.
\end{proof}

\subsection{The general $q$ case}
Let $\gamma$ be a strong composition of $n$. Recall that the weight of a tableau $T\in\NAT(\gamma)$ is given by:
\[
\wt(T) = t^{\coinv(T)} q^{\maj(T)} \prod_{\substack{u \in \widehat{\dg}(\gamma)\\T(\South(u))\neq T(u)}} \frac{1-t}{1-q^{\leg(u)+1}t^{\arm(u)+1}}
\]
As in the $q=0$ case, we use the destandardization map $\delta_S$ to split the set $\NAT(\gamma)$ into disjoint components indexed by their representative standard fillings in $\ST(\gamma)$. For $\tau\in\ST(\gamma)$ and any $V(\tau)\subseteq S \subseteq [n-1]$, again, the only cells that will potentially change the weight of the destandardized tableau $\delta_S(\tau)$ are the ones in $W(\tau)$. This is because when an entry $i\in \tau$ has $i+1$ above it, it is possible for that pair to destandardize to the same value if $\delta_S(\tau)\circ \tau^{-1}(i)=\delta_S(\tau)\circ \tau^{-1}(i+1)$, changing the product within the weight function.

We require the following lemma.

\begin{lemma}
Let $W$ be any subset of the cells of $\dg(\lambda)$. Then
\[
\prod_{u \in W} \frac{1-t}{1-q^{\leg(u)+1}t^{\arm(u)+1}} =\sum_{U \subseteq W} (-t)^{|U|} \left( \prod_{u \in U} \frac{1-q^{\leg(u)+1}t^{\arm(u)}}{1-q^{\leg(u)+1}t^{\arm(u)+1}} \right).
\]
\label{bino}
\end{lemma}
\begin{proof}
First we note that
\[
\frac{1-t}{1-q^{\leg(u)+1}t^{\arm(u)+1}} = 1+ \frac{-t+q^{\leg(u)+1}t^{\arm(u)+1}}{1-q^{\leg(u)+1}t^{\arm(u)+1}} = 1+ \frac{-t(1-q^{\leg(u)+1}t^{\arm(u)})}{1-q^{\leg(u)+1}t^{\arm(u)+1}}
\] 
Then we apply the binomial theorem. 
Given a set $S$ and a function $f:S\rightarrow \mathbb{Q}(q,t)$, the binomial theorem states
\[
\prod_{s\in S}(1+f(s))=\sum_{U\subseteq S}\left(\prod_{u\in U}f(u)\right).
\]
We apply this with $f(u)=\frac{-t\left(1-q^{\leg(u)+1}t^{\arm(u)}\right)}{1-q^{\leg(u)+1}t^{\arm(u)+1}}$.
\end{proof}

We can now prove the main result.
\begin{proof}[Proof of \cref{main}]
By \cref{lem:std,lem:delta} we have
\begin{align*}
G_{\gamma}(X;q,t) &= \sum_{T\in\NAT(\gamma)} \wt(T)x^T\\
&=\sum_{\tau\in\ST(\gamma)}\sum_{V(\tau)\subseteq S \subseteq [n-1]}  \wt(\delta_S(\tau)) M_{S}.\\
\end{align*}

\noindent
Applying \cref{lem:wtq},
\begin{align*}
G_{\gamma}(X;q,t) =& \sum_{\tau\in\ST(\gamma)} t^{\coinv(\tau)}q^{\maj(\tau)}\left( \prod_{\substack{u \in \widehat{\dg}(\gamma)\\u\not\in W(\tau)}} \frac{1-t}{1-q^{\leg(u)+1}t^{\arm(u)+1}} \right)\\
&
\times \sum_{V(\tau)\subseteq S \subseteq [n-1]} \left( \prod_{u \in S\cap W(\tau)} \frac{1-t}{1-q^{\leg(u)+1}t^{\arm(u)+1}} \right) M_{S}\\
=&\sum_{\tau\in\ST(\gamma)} t^{\coinv(\tau)}q^{\maj(\tau)}\left( \prod_{\substack{u \in \widehat{\dg}(\gamma)\\u\not\in W(\tau)}} \frac{1-t}{1-q^{\leg(u)+1}t^{\arm(u)+1}} \right)\\
&\times \sum_{W\subseteq W(\tau)}\left( \left( \prod_{u \in W} \frac{1-t}{1-q^{\leg(u)+1}t^{\arm(u)+1}} \right)
\sum_{\substack{V(\tau)\subseteq S \subseteq [n-1]\\S\cap W(\tau)=W}} M_{S}\right).
\end{align*}

\noindent
By \cref{bino},
\begin{align*}
&\sum_{W\subseteq W(\tau)}\left( \left( \prod_{u \in W} \frac{1-t}{1-q^{\leg(u)+1}t^{\arm(u)+1}} \right)
\sum_{\substack{V(\tau)\subseteq S \subseteq [n-1]\\S\cap W(\tau)=W}} M_{S}\right)\\
&=\sum_{W\subseteq W(\tau)}\left(\sum_{U \subseteq W} (-t)^{|U|} \left( \prod_{u \in U} \frac{1-q^{\leg(u)+1}t^{\arm(u)}}{1-q^{\leg(u)+1}t^{\arm(u)+1}} \right)\sum_{\substack{V(\tau)\subseteq S \subseteq [n-1]\\S\cap W(\tau)=W}} M_{S}\right).
\end{align*}
Next, we reorder the summations, as
\begin{align*}
&=\sum_{U \subseteq W(\tau)} \left((-t)^{|U|} \left( \prod_{u \in U} \frac{1-q^{\leg(u)+1}t^{\arm(u)}}{1-q^{\leg(u)+1}t^{\arm(u)+1}} \right)\sum_{\substack {W\subseteq W(\tau)\\ W\supseteq U}}\sum_{\substack{V(\tau)\subseteq S \subseteq [n-1]\\S\cap W(\tau)=W}} M_{S}\right).
\end{align*}
Finally, we gather terms and apply \eqref{eq:F},
\begin{align*}
&=\sum_{U \subseteq W(\tau)} \left((-t)^{|U|} \left( \prod_{u \in U} \frac{1-q^{\leg(u)+1}t^{\arm(u)}}{1-q^{\leg(u)+1}t^{\arm(u)+1}} \right)\sum_{V(\tau)\cup U \subseteq S \subseteq [n-1]} M_{S}\right)\\
&=\sum_{U \subseteq W(\tau)} (-t)^{|U|} \left( \prod_{u \in U} \frac{1-q^{\leg(u)+1}t^{\arm(u)}}{1-q^{\leg(u)+1}t^{\arm(u)+1}} \right)F_{V(\tau)\cup U}.
\end{align*}
The theorem follows.
\end{proof}

\section{Further simplifications and specializations}\label{sec:HL}

In this section, we further simplify the result for the Hall--Littlewood case ($q=0$) from \cref{sec:q0}. First, we introduce some notation. Let 
\begin{align*}
\ST_{1}(\gamma)= \{\tau \in \ST(\gamma)\colon i \in \Des(\tau) \implies \South(i) = i-1\}.
\end{align*}
That is, reading down columns, values can decrease by at most 1 per cell. 
We may send any element of $\tau' \in \ST_1(\gamma)$ to an element of $\tau \in \ST_0(\gamma)$ by sorting entries within their columns to become weakly decreasing from bottom to top. In this case the cells containing descents are sent to some $U\subset W(\tau)$, though the values in the cells may change. To make this sorting function invertible, we need to keep track of $U$. For any $U\subseteq W(\tau)$, consider the map $\iota_U$ that sends  $\tau \in \ST_0(\gamma)$ to $\tau' \in \ST_{1}(\gamma)$  by reversing the order of certain consecutive values in columns of $\tau$. 
Specifically, for each maximal set $\{i, i+1, \ldots, i+k-1\} \in U$, the values in $[i, i+k]$ are reversed so that $\{i+1, i, \ldots, i+k\}$ is similarly maximal in $\Des(\tau')$.
All other values are fixed. Note that since the cells of standard filling $\tau$ are identified with the values they contain, we represent $U$ by a subset of $[n]$.
Further, since $\tau\in\ST_0(\gamma)$ is the result of sorting the entries within the columns of $\tau'$, $\tau'$ has a unique preimage.

Next, for $\tau' \in \ST_{1}(\gamma)$,  let  
\begin{align*}
&\coinv(\Des(\tau')):= \sum_{u\in \Des(\tau')}\arm(u), \\
&\omega(\tau') := h(\gamma) - |\{i \in \tau' \colon i \text{ and } i+1 \text{ share a column}\}|.
\end{align*}

Lastly, we need to replace $V(\tau')$ with a new set in the context of $\ST_1(\gamma)$. The \emph{descent group of $i$} is the maximal connected set of cells in the column of $i$ such that every cell is a decent except the bottom cell. By construction, every cell is contained in a unique descent group. We say \emph{$i$ attacks $i+1$ through descents} if a cell in the descent group of $i$ attacks a cell in the descent group of $i+1$. Notice that if $i$ attacks $i+1$ in $\tau'\in \ST_1(\gamma)$, then $i$ must be at the top of its descent group and $i+1$ must be at the bottom of its descent group. Define $\Nu(\tau')$ as the set of $i$ in $\tau'$ such that $i\in\ID(\tau')$
or $i$ attacks $i+1$ through descents. Since $i$ attacking $i+1$ implies $i$ attacks $i+1$ by descents, it follows that $V(\tau')\subseteq  \Nu(\tau')$.

\begin{example}\label{example: HLstats}
Consider $\tau \in \ST_0((1,4,3))$ and $\tau' = \iota_{\{3, 4\}}(\tau) \in \ST_1((1,4,3))$.
\begin{center}
\begin{tikzpicture}[scale=0.5]
\node at (-2,-1) {$\tau=$} ;
\cell023
\cell124\cell111
\cell225\cell212
\cell327\cell318\cell306
\end{tikzpicture} 
\qquad \qquad
\begin{tikzpicture}[scale=0.5]
\node at (-2,-1) {$\tau'=$};
\cell025
\cell124\cell111
\cell223\cell212
\cell327\cell318\cell306
\end{tikzpicture} 
\end{center}
Here, $\coinv(\Des(\tau')) = 2$, $\omega(\tau') = 2$, $\ID(\tau')=\{3,4,7\}$, and $\Nu(\tau')=\{2, 3, 4, 5, 6, 7\}$.  
\end{example}

\begin{theorem*}[\cref{thm:HL}]
Let $\gamma$ be a strong composition. Then
\begin{align*}
G_{\gamma}(X;0,t)& =
\sum_{\tau\in\ST_{1}(\gamma)}
(1-t)^{\omega(\tau)}
(-t)^{|\Des(\tau)|}
t^{\coinv(\tau)-\coinv(\Des(\tau))} 
 F_{\Nu(\tau)}.
\end{align*}
\end{theorem*}
\begin{proof}
The proof is a matter of changing the order of summations in \eqref{eq:HL}, applying $\iota_U$, tracking the changes to statistics, and combining the sums. Reversing the order of summations gives

\begin{align}\label{eq: reversed sum}
G_{\gamma}(X;0,t) 
= \sum_{U \subseteq \widehat{\dg}(\gamma)}
\sum_{\tau\in\ST_0(\gamma)  \atop W(\tau)\supseteq U}  
(1-t)^{h(\gamma)-|W(\tau)|}
(-t)^{|U|}
t^{\coinv(\tau)} 
F_{V(\tau)\cup U}.
\end{align}

Next, apply $\iota_U$ to each $\tau$. The image is in $\ST_{1}(\gamma)$. Further, each tableau of  $\ST_1(\gamma)$ will appear exactly once. We now work through the changes to the factors in (\ref{eq: reversed sum}) from left to right.

Consider the image of $W(\tau)$ under $\iota_U$. Because tableaux in $\ST_0(\gamma)$ have no descents, $W(\tau)$ is precisely the set containing those $i$ that share a column with $i+1$. Because $\iota_U$ permutes entries within columns, $h(\gamma) - |W(\tau)|$ is taken to $\omega(\tau)$.

By the definition of $\iota_U$, $|U|$ is taken to $|\Des(\tau)|$. Thus, $(-t)^{|U|}$ is taken to $(-t)^{|\Des(\tau)|}$.

Notice that every type $B$ triple with its upper right cell contained in $U$ is changed from an inversion to a coinversion. We may compensate by multiplying by 
$t^{-\coinv(U)}$. Because the relative order of entries within all other triples is preserved, no other coinversions are changed. 

Finally, $\Nu(\iota_U(\tau)) = V(\tau)\cup U$. To see this, first notice that each $i \in U$ is moved below $i+1$, and so is moved after $i+1$ in the reading word. Hence, each element of $U$ is taken to an element of $\Nu(\iota_U(\tau))$. The only other values that need to be considered are those $i$ that are in a different column than $i+1$. In this case, $i$ can only move up relative to $i+1$, and $i+1$ can only move down relative to $i$. It is thus possible that $\iota_U$ removes $i$ from $V(\tau)$, but it cannot add any values to $V(\tau)$. By construction of $\iota_U$, if $i$ and $i+1$ were attacking in $\tau$ or $i\in\ID(\tau)$, 
then $i$ and $i+1$ are attacking through descents in $\iota_U(\tau)$. We thus replace $F_{V(\tau)\cup U}$ with $F_{\Nu(\tau)}$. Tracking these changes gives

\begin{align*}
G_{\gamma}(X;0,t) 
= \sum_{U \subseteq \widehat{\dg}(\gamma)}
\sum_{\tau\in\ST_{1}(\gamma)  \atop \Des(\tau) = U} 
(1-t)^{\omega(\tau)}
(-t)^{|\Des(\tau)|} 
t^{\coinv(\tau)-\coinv(\Des(\tau))} 
F_{\Nu(\tau)}.
\end{align*}
Combining the two sums completes the proof.
\end{proof}

\subsection{Jack specialization}
We also consider the specialization of $G_{\gamma}(X;q,t)$ to the setting of Jack polynomials, from which we immediately get a new definition of a \emph{quasisymmetric Jack polynomial}. Recall that the Jack polynomial indexed by a partition $\lambda$ with parameter $\alpha$ is a symmetric polynomial that can be obtained from 
\[
J_{\lambda}(X;\alpha)=\lim_{t\rightarrow1^{-}}
\left( \prod_{u\in\dg(\lambda)} \frac{1-t^{\arm(u)+1}q^{\leg(u)}}{1-t} \right)
P_{\lambda}(X;t^{\alpha},t).
\]

\noindent
Thus we define the quasisymmetric Jack polynomial indexed by a strong composition $\gamma$ as
\begin{equation}
G_{\gamma}(X;\alpha)=\lim_{t\rightarrow1^{-}} 
\left( \prod_{u\in\dg(\lambda)} \frac{1-t^{\arm(u)+1}q^{\leg(u)}}{1-t} \right)
G_{\gamma}(X;t^{\alpha},t).
\end{equation}

\noindent
Using \eqref{eq:G}, we get
\begin{equation}\label{eq:jack}
G_{\gamma}(X;\alpha)=\sum_{\sum_{\substack{T\in\NAT(\gamma)\\T\ \text{packed}}} } \left(\prod_{\substack{u \in \widehat{\dg}(\lambda)\\T(u)=T(\South(u))}} (\alpha(\leg(u)+1)+\arm(u)+1) \right)M_T. 
\end{equation}

\noindent
This is a  refinement of the Knop--Sahi formula \cite[Theorem 5.1]{KS96},
\[
J_{\lambda}(X;\alpha)=\sum_{\substack{T\in\NAT(\lambda)}}  \left(\prod_{\substack{u \in \widehat{\dg}(\lambda)\\T(u)=T(\South(u))}} (\alpha(\leg(u)+1)+\arm(u)+1) \right) x^T. 
\]

\noindent
Using \cref{thm:q} we obtain the following corollary.
\begin{corollary}
The quasisymmetric Jack polynomial has the following fundamental expansion:
\begin{align*}
G_{\gamma}(X;\alpha)=&\sum_{\tau\in\ST(\gamma)}  \left( \prod_{u\in W(\tau)} (\alpha(\leg(u)+1)+\arm(u)+1) \right) \\
&\times \sum_{U \subseteq W(\tau)} (-1)^{|U|} \left( \prod_{u \in U} \frac{\alpha(\leg(u)+1)+\arm(u)}{\alpha(\leg(u)+1)+\arm(u)+1} \right)F_{V(\tau)\cup U}.
\end{align*}
\end{corollary}

\section*{acknowledgement}
The authors would like to thank Jim Haglund for bringing them together. The first and second authors thank the Institute of Mittag-Leffler for their hospitality.
The second author was supported by NSF grants DMS-1704874 and DMS-1953891.

\printbibliography

\end{document}